\documentclass[12pt]{amsart}

\usepackage{tikz}
\usetikzlibrary{snakes}
\usepackage{amsmath,amssymb,amsthm,amscd}
\usepackage{comment}
\usepackage{multirow}
\usepackage{mathtools}
\usepackage{enumitem}
\usetikzlibrary{arrows}
\usepackage[normalem]{ulem}

\usepackage[margin=1in]{geometry} 

\numberwithin{equation}{section}

\newtheorem{theorem}{Theorem}[section]
\newtheorem{proposition}[theorem]{Proposition}
\newtheorem{lemma}[theorem]{Lemma}

\newtheorem*{theorem*}{Theorem}

\theoremstyle{definition}

\newtheorem{example}[theorem]{Example}
\newtheorem{remark}[theorem]{Remark}

\renewcommand\P{\mathbb P}
\newcommand\adim{\operatorname{adim}}
\newcommand\edim{\operatorname{edim}}
\newcommand\vdim{\operatorname{vdim}}

\definecolor{MyDarkGreen}{cmyk}{1,0,1,0}

\def\cocoa{{\hbox{\rm C\kern-.13em o\kern-.07em C\kern-.13em o\kern-.15em A}}}

\begin{document}

\title{New constructions of unexpected hypersurfaces in $\P^n$}

%\title[Cones]{Producing unexpected hypersurfaces from cones }

\author{Brian Harbourne}
\address{Department of Mathematics\\
University of Nebraska\\
Lincoln, NE 68588-0130 USA}
\email{bharbourne1@unl.edu}

\author{Juan Migliore} 
\address{Department of Mathematics \\
University of Notre Dame \\
Notre Dame, IN 46556 USA}
 \email{migliore.1@nd.edu}

\author{Halszka Tutaj-Gasi\'nska}
\address{Faculty of Mathematics and Computer Science\\
Jagiellonian University\\
{\L}ojasiewicza 6, PL-30-348 Krak\'ow, Poland}
\email{Halszka.Tutaj@im.uj.edu.pl}

\begin{abstract} 
In the paper we present new examples of unexpected varieties. The research on unexpected varieties started 
with a paper of Cook II,  Harbourne, Migliore and Nagel and was
continued in the paper of  Harbourne, Migliore,  Nagel and Teitler. Here we present three ways of 
producing unexpected varieties that expand on what was previously known.
In  the paper of  Harbourne, Migliore,  Nagel and Teitler,
cones on varieties of codimension 2 were used to produce unexpected hypersurfaces.
Here we show that cones on positive dimensional varieties of codimension 2 or more
almost always give unexpected hypersurfaces. For non-cones, almost all previous
work has been for unexpected hypersurfaces coming from finite sets of points.
Here we construct unexpected surfaces coming from lines in $\P^3$,
and we generalize the construction using birational transformations
to obtain unexpected hypersurfaces in higher dimensions.
\end{abstract}

\date{April 5, 2019}

\thanks{
{\bf Acknowledgements}: Harbourne was partially supported by Simons Foundation grant \#524858.
Migliore was partially supported by Simons Foundation grant \#309556.
Tutaj-Gasi\'nska was partially supported by National Science Centre grant 2017/26/M/ST1/00707.
Harbourne and Tutaj-Gasi\'nska thank the Pedagogical University of Cracow, the Jagiellonian University 
and the University of Nebraska for hosting reciprocal visits by Harbourne and Tutaj-Gasi\'nska
when some of the work on this paper was done.}



\maketitle

%\tableofcontents

%%%%%%%%%%%%%%%%%%%%%%%%%%%%%%%%%%%%%%%%%%%%%%%%%%%%%%%%%%%%%%

\section{Introduction}
We work over an arbitrary algebraically closed field $K$ except for results that depend
on computer calculations, and for these we assume characteristic 0. The notion of an unexpected variety was introduced in \cite{CHMN16}.
That paper, using the results of \cite{DIV}, produced an example of a quartic curve in $\P^2$
which passes through a certain special set $B$ of nine points imposing independent conditions on quartics. 
Thus the space of quartics passing through $B$ has (affine) dimension 6; however, for any general point $P$ 
on $\P^2$, there exists a quartic passing through $B$ and vanishing at $P$ with multiplicity 3. 
This is unexpected since vanishing at a triple point typically imposes 6 conditions, so there should be no
quartic vanishing on $B$ and triply at $P$. 

In general, let $R=K[\P^n]=K[x_0,\ldots,x_n]$ denote the homogeneous coordinate ring of $\P^n$; it is a polynomial ring
in $n+1$ indeterminates with the standard grading in which each variable $x_i$ has degree 1. Then,
given a scheme $Z\subseteq \P^n$, we denote by $I_Z\subset R$ the saturated homogeneous ideal defining $Z$. 
Now let $\lambda_i\subset \P^n$, $1\leq i\leq r$, be general linear varieties and let $\delta_i=\dim \lambda_i$.
For integers $m_i\geq0$, $X=m_1\lambda_1+\cdots+m_r\lambda_r$ denotes the scheme defined by the
ideal $I_X=\cap_i I_{\lambda_i}^{m_i}\subseteq R$. 

We denote the Hilbert polynomial of $R/I_X$ by $H_X$.
We denote by $[R]_t$ the $K$-vector space span of all forms in $R$ of degree $t$.
Given a homogeneous ideal $I\subseteq R$ we denote by $[I]_t$ the sub-$K$-vector space
spanned by all forms in $I$ of degree $t$. Thus $H_X(t)=\dim [R]_t-\dim [I_X]_t$.
Given a linear variety $\lambda\subseteq \P^n$ of dimension $\delta$,
we set $c_{n,\delta,m,t}= H_{m\lambda}(t)$ and say that vanishing on $\lambda$ to order $m$ imposes
$c_{n,\delta,m,t}$ conditions on forms of degree $t$. If $\lambda_1,\ldots,\lambda_r$ are pairwise disjoint,
then $H_X(t)=\sum_ic_{n,\delta_i,m_i,t}$.

Now let $Z\subseteq \P^n$ be a scheme and let $X=\sum_i m_i\lambda_i$
where $\lambda_i\subset \P^n$, $1\leq i\leq r$, are general linear varieties
of dimension $\delta_i$ which are disjoint from $Z$ and where each $m_i>0$ is an integer. 
We say $(n,X,Z,t)$ (or just $(X,Z,t)$ if $n$ is understood) is {\it unexpected} if
$$\dim [I_{X\cup Z}]_t > \max(0, \dim [I_Z]_t - H_X(t));$$
i.e., $(n,X,Z,t)$ is unexpected if vanishing on $X$ imposes on $[I_Z]_t$ fewer than the expected
number of conditions (namely, $H_X(t)$). 
We refer to the varieties defined by the forms $[I_{X\cup Z}]_t$ as being {\it unexpected hypersurfaces}
for $Z$ (with respect to $X$ in degree $t$). 
In the example above of plane quartics vanishing
on $B$ with a general triple point at a general point $P$, we thus have $(2,3P,B,4)$ being unexpected
and we say that the quartics vanishing on $B$ with a general triple point are unexpected curves for $B$.

In \cite{CHMN16} the focus was on unexpectedness of $(n,X,Z,t)$ where $n=2$, $Z\subset\P^2$ 
was a finite reduced set of points and $X=mP$ for a general point $P\in\P^2$ with $m=t-1$.
Since \cite{CHMN16}, more papers on unexpected varieties have appeared, including \cite{BMSS}, \cite{HMNT} and 
\cite{DHRST18}, among others: for \cite{BMSS}, $(n,X,Z,t)$ has $2\leq n\leq 3$ with $X=mP$, $m=t-1$ and $Z$ a finite 
reduced set of points; for \cite{HMNT}, $(n,X,Z,t)$ has $n$ arbitrary, $X=mP$ for a general point $P$ with $t\geq m$
and $Z$ a reduced set of points; and for \cite{DHRST18}, $(n,X,Z,t)$ has $n=3$, $X=\sum m_i\lambda_i$ 
where the $\lambda_i$ are general lines and $Z=\varnothing$.
In \cite{HMNT} in particular, a method of producing unexpected surfaces in $\P^n$ is 
described which is based on building cones over codimension 2 subvarieties in $\P^n$. The paper \cite{DHRST18} considers 
four surfaces in $\P^3$ vanishing on lines with multiplicities; these surfaces, by dimension count, should not exist, 
but \cite{DHRST18} shows they do exist and are unique and irreducible.

The present paper builds on the developments of \cite{HMNT} and \cite{DHRST18}. 
In \S\ref{Cones} we generalize the cone method of \cite{HMNT} to obtain unexpected varieties in 
$\P^n$ where $X=m\lambda$ for a positive dimensional linear variety $\lambda$; here the unexpected varieties
are cones with apex $\lambda$.
(Given an equidimensional variety $V\subset\P^n$ of dimension $e$ and a disjoint linear variety
$\lambda$ of dimension $n-e-2$, recall that the cone $C_\lambda(V)$ over $V$ with apex (or vertex or axis)
$\lambda$ is the union of all lines through a point of $\lambda$ and a point of $V$.
An easy argument using projection from $\lambda$ to a complementary linear subvariety shows that $C_\lambda(V)$
is Zariski closed; see, for example, the proof of Lemma \ref{lemma2}.) The main theorem in \S\ref{Cones} is 
Theorem \ref{MainThmSect2}, which is as follows:
 
\begin{theorem}
Let $V \subset \P^n$ be a reduced, equidimensional, non-degenerate subvariety of dimension $e$ ($1 \leq e \leq n-2$) and 
degree $d$. Let $\lambda$ be a general linear space of dimension $n-2-e$ (i.e., of codimension $e+2$). 
Then $C_\lambda(V)$ is the unique hypersurface of degree $d$ vanishing to order $d$ on $\lambda$ and containing 
$V$, and $(n,d\lambda,V,d)$ is unexpected.
\end{theorem}

In \S\ref{LinesInP3} we introduce a notion $u$ of unexpectedness, measuring how much the actual dimension 
of a system (of forms of degree $d$ vanishing on a given subscheme $Z$) differs from the virtual dimension. 
We prove the following theorem.

\begin{theorem}\label{Thm2}
Let $L_1,\ldots,L_r,L_1',\ldots,L_s'$ be general lines in $\P^3$.
Let $X=m_1L_1+\cdots+m_rL_r$ and $Z'=m_1'L_1'+\cdots+m_s'L_s'$.
If $u(X+Z',\varnothing,t)>u(Z',\varnothing,t)$, then 
$(X,Z',t)$ is unexpected (i.e., $u(X,Z',t)>0$), in which case so is $(X,Z,t)$,
where $Z$ is obtained from $Z'$ by, for each $i$, either keeping $m_i'L_i'$ as is or replacing it by
a fat point subscheme consisting of any $s_i\geq t-m_i'+2$
points of $L_i'$, each having an assigned multiplicity of at most $m_i'$ but where
at least $t-m_i'+2$ of the $s_i$ points on $L_i'$ have multiplicity exactly $m_i'$.
\end{theorem}
\noindent

We apply Theorem \ref{Thm2} using the results of \cite{DHRST18} to give examples of unexpected surfaces with multiple lines in $\P^3$.

In \S\ref{Veneroni}, the last section of the paper, we use certain birational transformations 
of $\P^n$, called Veneroni transformations,
to construct examples of unexpected hypersurfaces in $\P^n$ for $n=3,4,5$.
For $n=4$ and $5$, the linear spaces of which $X$ is composed are not disjoint, in contrast to
our constructions in sections 2 and 3.

\section{Cones}\label{Cones}
Throughout this section we denote by $V \subset \P^n$ a reduced, equidimensional, non-degenerate subvariety of dimension 
$e$ ($1 \leq e \leq n-2$) and degree $d$ in $\P^n$. Furthermore, we denote by $\lambda$ a general linear space of dimension $n-e-2$ (i.e. of codimension $e+2$). Experiments leading to the results in this section were obtained using CoCoA \cite{CoCoA}.

\begin{lemma}[\cite{DHST14} Lemma A.2(C)] \label{adv}
Let $L$ be a linear subspace of $\P^n$ of dimension $\delta<n$. Let $t \geq m$ be positive integers. Then vanishing to order at least $m$ along $L$ imposes exactly
\[
c_{n,\delta,m,t} = \sum_{0 \leq i < m} \binom{t-i+\delta}{\delta} \binom{i+n-\delta-1}{n-\delta-1}
\]
linearly independent conditions on forms of degree $t$. When $m=t$, this gives
\[
c_{n,\delta,m,m} = \binom{m+n}{n} - \binom{m+n-\delta-1}{n-\delta-1},
\]
while for $\delta=0$ it gives
$$c_{n,0,m,t}=\binom{n+m-1}{n}$$ 
and for $\delta=1$ it gives
$$c_{n,1,m,t}=(t+1)\binom{m+n-2}{n-1}-(n-1)\binom{m+n-2}{n}.$$ 
\end{lemma}

\begin{remark}
We denote by $d\lambda$ the subscheme of $\P^n$ defined by the ideal $I_\lambda^d$, which it is easy to see is saturated. 
With $V$ and $\lambda$ as above, applying 
Lemma \ref{adv} with $(n,\delta,m,t)=(n,n-e-2,d,d)$ allows us to write the number of conditions that we expect $d\lambda$ to impose on 
$[I_V]_d$, namely
\[
c_{n,n-e-2,d,d} =\binom{d+n}{n} - \binom{d+e+1}{e+1}.
%= \binom{d+e+1}{e+2} + \binom{d+e+2}{e+3} + \dots + \binom{d+n-1}{n}
\]
\end{remark}

We now consider the cone over $V$ with axis $\lambda$.

\begin{lemma} \label{lemma2}
Let $V$ be as above, let $p_1,\ldots,p_i$, $i=n-e-1$, be general points of $\P^n$ and let $\lambda$
be the span of the points $p_j$ (so $\lambda$ is itself general). 
%Let $V' \subset V$ be the set of smooth points of $V$. 
Let $S$ be 
%the closure of 
the union of lines $\overline{PQ}$, where $P \in V$ and $Q \in \lambda$. Then $S=C_\lambda(V)$ is a hypersurface of degree $d = \deg V$. 
\end{lemma}

\begin{proof}
Let $\beta:\P^n_\lambda\to \P^n$ be the blow up of $\P^n$ along $\lambda$ and let $\pi:\P^n_\lambda\to \P^{e+1}$
be the projection from $\lambda$ to a general linear subspace $\P^{e+1}\subset\P^n$ of complementary dimension.
(So, identifying $\P^n \backslash \lambda$ as an open subset of $\P^n_\lambda$,
if $P \in \P^n \backslash \lambda$ then 
$\pi(P) = \hbox{span}(P,\lambda) \cap \P^{e+1}$.)
Then $S=\beta(\pi^{-1}(\pi(V)))$ and $\dim \pi(V')=e$ for each component $V'$ of $V$.
(To see $\dim \pi(V')=e$, note that we can regard $\pi:\P^n \backslash \lambda\to \P^{e+1}$
as a composition of a sequence of projections from the points $p_j$. Take 
a flag $\P^{e+1}=\P_i\subsetneq\P_{i-1}\subsetneq \cdots\subsetneq \P_1\subsetneq \P_0=\P^n$
of linear spaces $\P_j$ such that $p_1\in\lambda\cap\P_0$, $p_2\in\lambda\cap\P_1$, 
$\ldots$, $p_i\in\lambda\cap\P_{i-1}$, and define $\pi_j: \P_{j-1} \backslash \lambda\to \P_j$
to be the projection of $\P_{j-1}\backslash \lambda$ from $p_j$ to $\P_j$.
Since each $\pi_j$ has one dimensional fibers, if $W\subset \P_{j-1}$ is closed and irreducible 
with $W\cap \lambda=\varnothing$, then $\pi_j|_W$ is finite so not only is $\pi_j(W)$ irreducible and closed in $\P_j$,
but $\dim \pi_j(W)=\dim W$. Thus $\dim \pi(V')=\dim V'=e$ under the composition of the $\pi_j$.) 

Since $\pi$ has fibers of dimension $(n-e-2)+1$, each component of 
$\pi^{-1}(\pi(V))$ has dimension $e+(n-e-2)+1=n-1$. 
Since $\beta$ is birational off $\lambda$, $\dim S=\dim \pi^{-1}(\pi(V))=n-1$.
Let $V_j$ be the image of $V$ in $\P_j$ under the composition of the projections 
$\pi_1,\ldots,\pi_j$. 
Since for any point $p\in V_j$, the cone over $V_j$ with apex $p$ has dimension at most
$1+\dim V_j<\dim \P_j$, projection from a general point in $\P_j$ will be generically injective.
Thus $\deg V_j=\deg V$ for all $j$. In particular, $\deg \pi(V)=\deg V=d$.
Since $\pi$ maps a general line $L$ in $\P^n$ to a general line $L'$ in $\P_i$, $L'$ meets
$\pi(V)$ in $d$ distinct points. Thus $L$ meets $S$ in $d$ points so $\deg(S)=d$.
\end{proof}

\begin{lemma}\label{MultLemma}
Let $V, \lambda$ and $S$ be as in Lemma \ref{lemma2}. Then $S$ has multiplicity $d$ along $\lambda$. 
\end{lemma}

\begin{proof}
%Let $Q$ be a general point of $\lambda$ and $P$ a general point of any component of $S$. We first claim that the line 
%$\mu = \overline{PQ}$ lies on $S$. Since $P \in S$, there is  a point $R \in V$ such that the linear space $\Lambda$ 
%spanned by $R$ and $\lambda$ contains $P$. We also have $Q \in \Lambda$, so $\mu \subset \Lambda \subset S$ 
%as desired. This shows that $S$ has multiplicity at least $d$ along $\lambda$. But since $\deg S = d$, we must have multiplicity exactly $d$.
Let $Q$ be a general point of $\lambda$ and $P$ a general point of $\P^n$. Let $L$ be the line defined by $P$ and $Q$.
With $\beta$ and $\pi$ as in the proof of Lemma \ref{lemma2}, then $\pi(P)\not\in\pi(V)$ and
$L\subset\beta^{-1}(\pi(P))$, so $L$ meets $S$ only at $Q$. Thus $S$ has multiplicity $d$ at $Q$.
\end{proof}

\begin{proposition} \label{curve}
Let $C$ be a reduced, equidimensional, non-degenerate curve of degree $d$ in $\P^n$.  Let $\lambda$ be a general 
linear space of dimension $n-3$ in $\P^n$. Let $S$ be the hypersurface defined in Lemma \ref{lemma2}. Then 
$(n,d\lambda,C,d)$ is unexpected, hence $S$ is unexpected for $C$.
\end{proposition}

\begin{proof}
Using Lemma \ref{adv}, it is enough to show
\[
\dim [I_C]_d - \left [ \binom{d+n}{n} - \binom{d+2}{2} \right ] \leq 0.
\]
By \cite[Remark 1, p. 497]{GLP} the map
\[
H^0(\mathcal O_{\P^n}(d)) \rightarrow H^0(\mathcal O_C(d))
\]
is surjective. 
We also know that $h^0(\mathcal O_C(d)) = d^2 - g_C + 1$ and that $g_C<g_D$, where $g_C$ is the arithmetic 
genus of $C$ and where $D$ is a plane curve of degree $d$ (see the proof of \cite[Proposition 2.1]{HMNT}). Hence
\[
\begin{array}{rcl}
\displaystyle \dim [I_C]_d - \left [ \binom{d+n}{n} - \binom{d+2}{2} \right ] & = & 
\displaystyle  \binom{d+n}{n} - [d^2 - g_C + 1] - \left [ \binom{d+n}{n} - \binom{d+2}{2} \right ] \\ \\
& = & \displaystyle  \binom{d+2}{2} - [d^2 - g_C + 1] \\ \\
& < & \displaystyle  \binom{d+2}{2} - [d^2 - g_D + 1] \\ \\
& = & 1.
\end{array}
\]
\end{proof}

\begin{theorem}\label{MainThmSect2}
Let $V \subset \P^n$ be a reduced, equidimensional, non-degenerate subvariety of dimension $e$ ($1 \leq e \leq n-2$) and 
degree $d$. Let $\lambda$ be a general linear space of dimension $n-e-2$ (i.e. of codimension $e+2$). Let $S$ be the 
corresponding union of lines as above (so $S=C_\lambda(V)$). Then 
$S$ is the unique hypersurface of degree $d$ vanishing to order $d$ on $\lambda$ and containing $V$, and
$(n,d\lambda,V,d)$ is unexpected, hence $S$ is unexpected
for $V$.
\end{theorem}

\begin{proof}
Since $S=C_\lambda(V)$, $S$ contains $V$.
By Lemmas \ref{lemma2} and \ref{MultLemma}, $S$ is a hypersurface of degree $d$ vanishing to order $d$ on $\lambda$.
But any hypersurface of degree $d$ containing $V$ and vanishing to order $d$ along $\lambda$ must contain each line
from $\lambda$ to $V$, hence any such hypersurface must contain $S$, and thus must be $S$; this shows uniqueness.
From Lemma \ref{adv} we know that the expected number of conditions imposed by $d\lambda$ is
\[
\binom{d+n}{n} - \binom{d+e+1}{e+1}.
\]
We want to show that
\[
\dim [I_V]_d - \left [ \binom{d+n}{n} - \binom{d+e+1}{e+1} \right ] \leq 0.
\]
Our proof is by induction on $e$. The case $e=1$ is Proposition \ref{curve}, so assume $e \geq 2$. Let $H$ be a general 
hyperplane and $W = V \cap H \subset H = \P^{n-1}$. Note that $\dim W = e-1$. We have the exact sequence 
\[
\begin{array}{ccccccccccccccccccc}
0 & \rightarrow & [I_V]_{d-1} & \rightarrow & [I_V]_d & \rightarrow & [I_{W|H} ]_d & \longrightarrow & H^1(\mathcal I_V(d-1)) & \rightarrow & \dots \\
&&&&&&\hbox{\ \ \ \ \ \ \ } \searrow & & \nearrow \hbox{\ \ \ \ \ \ \ \ \ \ \ \ } \\
&&&&&&& [K]_d \\
&&&&&&\hbox{\ \ \ \ \ \ \ } \nearrow & & \searrow \hbox{\ \ \ \ \ \ \ \ \ \ \ \ } \\
&&&&&&\hbox{\  } 0 & & 0 \hbox{\ \ \ \ \ \  } \\
\end{array}
\]
so
\[
\dim [I_V]_d = \dim [I_V]_{d-1} + \dim [I_{W|H}]_d - \dim [K]_d.
\]
Then 
\[
\begin{array}{l}
\displaystyle \dim [I_V]_d - \left [ \binom{d+n}{n} - \binom{d+e+1}{e+1} \right ]  \\ \\
\displaystyle =  \dim [I_V]_{d-1} + \dim [I_{W|H}]_d - \dim [K]_d -  \left [ \binom{d+n}{n} - \binom{d+e+1}{e+1} \right ] \\ \\
\displaystyle \leq \dim [I_V]_{d-1} + \binom{d+n-1}{n-1} - \binom{d+e}{e} - \dim [K]_d - \binom{d+n}{n} + \binom{d+e+1}{e+1} \\ \\
\displaystyle \leq \dim [I_V]_{d-1} - \binom{d+n-1}{n} + \binom{d+e}{e+1}
\end{array}
\]
(using induction on the third line). As in \cite{HMNT} Remark 2.2, for any $f \leq d-1$ there can be no hypersurface of 
degree $f$ containing $V$ and having multiplicity $f$ along $\lambda$ because then such a hypersurface would contain $S$, 
which is impossible since $\deg S = d > f$. This implies that the number of independent conditions $d\lambda$ imposes on 
$[I_V]_{d-1}$ is at least $\dim [I_V]_{d-1}$, so 
(using Lemma \ref{adv} and taking $m = t  = d-1$, $\delta = n-e-2$)
\[
\dim [I_V]_{d-1} \leq \binom{d-1+n}{n} - \binom{d-1+n - (n-e-2) - 1}{n-(n-e-2)-1} = \binom{d+n-1}{n} - \binom{d+e}{e+1}
\] 
hence $\dim [I_V]_d - \left [ \binom{d+n}{n} - \binom{d+e+1}{e+1} \right ] \leq 0$ and we are done.
\end{proof}

\begin{remark}
Given a nonzero vector subspace $V\subseteq (k[\P^n])_t$ with $t>0$, a general point of multiplicity 1 imposes 
exactly one condition on $V$.
It is natural to ask the analogous question of whether a general (reduced) line always imposes the expected number
of conditions (namely $t+1$). We now give a simple example in $\P^4$  to show that this is not the case.

Let $V=2L_1+L_2$, where $L_1, L_2, L_3$ are general lines in $\P^4$.  Then $\dim [I_{L_1}^2]_2 = 6$ 
and $\dim [I_{L_1}^2\cap I_{L_2}]_2 = 3$ as expected. If $L_3$ were 
to impose the expected 3 conditions, there would be no hypersurface of degree 2 containing $2L_1 + L_2 + L_3$. 
However, hyperplane the $H$ spanned by $L_1$ and $L_2$ and the hyperplane spanned by $L_1$ and $L_3$
together contain $2L_1 + L_2 + L_3$. Thus $L_3$ fails to impose the expected 3 conditions on $[I_V]_2$.
(Here it is easy to see why the number of conditions is less than 3, and thus why, at least in some cases,
counting conditions imposed by general lines is more complicated than doing so for general points. 
The hyperplane $H$ is a fixed component of the linear system
$[I_{L_1}^2\cap I_{L_2}]_2$ and $L_3$ intersects $H$; this point of intersection
reduces the number of conditions imposed by $L_3$.)
%This has an interesting consequence. If $V$ is a vector subspace of $(k[P^4])_2$ of dimension $\geq 3$, 
%we know that two general points impose exactly two conditions on $V$. A third general point imposes a 
%third independent condition. However, if instead the third  point is collinear with the first two (even generally 
%chosen on this line), it does not necessarily impose a new condition. This is in contrast to the fact that it does impose a new condition if $V = (k[P^4])_2$.

Example \ref{NiceLineExample} gives a more subtle and less easily explained instance of 
the conditions imposed by a reduced general line failing to be independent, but this time in $\P^3$.
\end{remark}

\section{Lines in $\P^3$}\label{LinesInP3}

Let $L_1,\ldots,L_r$ be general linear subvarieties of $\P^n$ where $d_i=\dim L_i$.
Let $X(m_1,\ldots,m_r)=m_1L_1+\cdots+m_rL_r$ be the scheme defined by the ideal $\cap_i I_{L_i}^{m_i}\subset K[\P^n]$.
When $d_i=\delta$ for all $i$ and some $\delta$, to efficiently indicate how often each multiplicity is repeated,
we will use $X_\delta(m_1^{\times r_1},\ldots,m_s^{\times r_s})$ to denote that there are $r_i$ general $L_i$ 
of dimension $\delta$ with multiplicity $m_i$
(with $m_i^{\times r_i}$ written as $m_i$ when $r_i=1$).

Let $Z$ be any subscheme of $\P^n$ and let $X=X(m_1,\ldots,m_r)$.
When $Z=\varnothing$, we regard $I_Z=K[\P^n]$.
Given the triple $(X,Z,t)$, where $t\geq0$ is an integer, we introduce the following notation:\vskip\baselineskip
$\adim (X,Z,t)=\dim [I_X\cap I_Z]_t$ is the {\it actual dimension} of $(X,Z,t)$;\vskip\baselineskip
$\vdim (X,Z,t)=\dim [I_Z]_t-H_X(t)$ is the {\it virtual dimension} of $(X,Z,t)$; and\vskip\baselineskip
$\edim (X,Z,t)=\max(0,\vdim(X,Z,t))$ is the {\it expected dimension} of $(X,Z,t)$,\vskip\baselineskip
\noindent and we have
$$\adim (X,Z,t)\geq\edim (X,Z,t)\geq\vdim (X,Z,t).$$
\vskip\baselineskip

We now define $u(X,Z,t)$, the {\it unexpectedness} of $(X,Z,t)$, as 
$$u(X,Z,t)=
\begin{cases}
 	 0, \textrm{ \ \ if }\adim(X,Z,t)=0\\
 	\adim(X,Z,t)-\vdim(X,Z,t), \textrm{ \ \ otherwise. }
\end{cases}
$$
Thus $(X,Z,t)$ is unexpected if and only if $u(X,Z,t)>0$.
%Equivalently, $(X,Z,t)$ is unexpected if $$\adim(X,Z,t) > \edim(X,Z,t).$$

An important special case is when $L_1,\ldots,L_r$ are pairwise disjoint,
because then we can be more explicit about the Hilbert polynomial $H_X$
since $H_X(t)=\sum_{1\leq i\leq r} c_{n,\delta,m_i,t}$.
Of course, there are no nontrivial forms of degree $t$ vanishing on 
$X(m_1,\ldots,m_r)$ when $t<m_i$ for some $i$. 
For $t\geq \max_i\{m_i\}$, we can apply Lemma \ref{adv} to compute $c_{n,\delta,m_i,t}$.
When $\delta=0$, it gives the well known fact that
$$c_{n,0,m,t}=\binom{n+m-1}{n}$$ 
and for $\delta=1$ it gives
$$c_{n,1,m,t}=(t+1)\binom{m+n-2}{n-1}-(n-1)\binom{m+n-2}{n}.$$ 
So, for example, 
$$c_{3,1,m,t}=\frac{1}{6}m(m+1)(3t+5-2m);$$ 
i.e., $\dim [I_{mL}]_t=\binom{d+3}{3}-c_{3,1,m,t}$ holds
when $t\geq m$ and $L$ is a line in $\P^3$.

We begin with a proposition that will be useful for our main focus of lines in $\P^3$, and which we use 
in the proof of Theorem \ref{extendedThm}.

\begin{proposition}\label{BezoutTypeProp} Let $p_1,\ldots,p_s\in L\subset \P^3$, where $L$ is a line and let
$F\in K[\P^3]$ be a form of degree $t$. Assume $F$ vanishes at each $p_i$ to order at least 
$m$. If $s\geq t-m+2$, then $F$ vanishes to order at least $m$ on $L$.
\end{proposition}

\begin{proof}
If $t<m$, then $F=0$ and the claim holds. Also, the claim clearly holds if $m=0$.
So assume $t\geq m>0$, so $t=m+\epsilon$ for some $\epsilon\geq 0$. 
The order of vanishing of $F$ on $L$ equals the order of vanishing of $F|_H$ on $L$, where
$H$ is a general plane containing $L$. 
Then the number of zeros of $F|_H$ on $L$ is 
at least $sm\geq (t-m+2)m=2m+\epsilon m >m+\epsilon=t$, so $F$ vanishes on $L$.
Since $F|_H$ vanishes on $L$, we see $F|_H-L$ vanishes at $s$ points of $L$ to order
at least $m-1$, but $F|_H-L$ has degree $t-1\geq m-1$, and we have $s\geq t-m+2=(t-1)-(m-1)+2$.
If $m-1>0$, arguing as before shows that $F|_H-L$ vanishes on $L$. 
We can continue in this way until $m$ is reduced to 0; thus $F|_H$ vanishes to order 
at least $m$ on $L$, and hence so does $F$.
\end{proof}

\begin{theorem}\label{extendedThm}
Let $L_1,\ldots,L_r,L_1',\ldots,L_s'$ be general lines in $\P^3$.
Let $X=m_1L_1+\cdots+m_rL_r$ and $Z'=m_1'L_1'+\cdots+m_s'L_s'$.
If $u(X+Z',\varnothing,t)>u(Z',\varnothing,t)$, then 
$(X,Z',t)$ is unexpected, in which case so is $(X,Z,t)$,
where $Z$ is obtained from $Z'$ by, for each $i$, either keeping $m_i'L_i'$ as is or replacing it by
a fat point subscheme consisting of $s_i$
points of $L_i'$ each having an assigned multiplicity of at most $m_i'$ but where
at least $t-m_i'+2$ of the points have multiplicity exactly $m_i'$.
\end{theorem}

\begin{proof}
By Proposition \ref{BezoutTypeProp}, $[I(Z')]_t=[I(Z)]_t$, hence if
$(X,Z',t)$ is unexpected, then so is $(X,Z,t)$.

To prove the rest of the statement of the theorem, say $u(X+Z',\varnothing,t)>u(Z',\varnothing,t)$.
Then $u(X+Z',\varnothing,t)>0$, so $\adim (X,Z',t)=\adim (X+Z',\varnothing,t)>0$.
Moreover,
$$\dim [I_{X+Z'}]_t-\left(\dim [K[\P^3]]_t -\sum_{1\leq i\leq r} c_{3,1,m_i,t}-\sum_{1\leq i\leq s} c_{3,1,m_i',t}\right)$$
$$=u(X+Z',\varnothing,t)>u(Z',\varnothing,t)=\dim [I_{Z'}]_t-\left(\dim [K[\P^3]]_t -\sum_{1\leq i\leq s} c_{3,1,m_i',t}\right)$$
so
$$u(X,Z',t)=\dim [I_{X+Z'}]_t-\left(\dim [I_{Z'}]_t-\sum_{1\leq i\leq r} c_{3,1,m_i,t}\right)>0;$$
i.e., $(X,Z',t)$ is unexpected.
\end{proof}

In order to have examples where we can apply Theorem \ref{extendedThm}, we recall \cite[Theorem 3.3]{DHRST18}:

 \begin{theorem}\label{ex:Singular list}
%\label{thm:main}
   Let $X\subset\P^3$ be a union of general lines with multiplicity. Then 
   the triple $(X,\varnothing,t)$ is unexpected in the following cases:
   \begin{itemize}
   \item[A)] $X=X_{1}(3^{\times 4},1^{\times 5})$ and $t=10$ (here $\adim=1$ and $\vdim=-1$);
   \item[B)] $X=X_{1}(4,3^{\times 5})$ and $t=12$ (here $\adim=1$ and $\vdim=-5$);
   \item[C)] $X=X_{1}(3^{\times 6},2)$ and $t=12$ (here $\adim=1$ and $\vdim=-2$); and
   \item[D)] $X=X_{1}(6^{\times 5},1)$ and $t=20$ (here $\adim=1$ and $\vdim=-105$).
   \end{itemize}
\end{theorem}

Since for $t\geq m$ and a linear variety $L$, we have (by Lemma \ref{adv}) that $u(mL,\varnothing,t)=0$, we get many new examples 
of unexpected triples by applying Theorem \ref{extendedThm} to Theorem \ref{ex:Singular list}
in case $Z'$ is a single line in $\P^3$.

\begin{example}
Consider general lines $L_{1,1},\ldots,L_{1,4},L_{2,1},\ldots,L_{2,5}$ in $\P^3$.
Let $X=3L_{1,1}+\cdots+3L_{1,4}+L_{2,1}+\cdots+L_{2,4}$, $Z'=L_{2,5}$ and
let $Z$ be any 11 or more points of $L_{2,5}$ of multiplicity 1.
Then by Theorem \ref{ex:Singular list}(A) we have $2=\adim-\vdim=u(X+Z',\varnothing,10)>u(Z',\varnothing,10)=0$, so
$(X,Z,10)$ and $(X,Z',10)$ are unexpected by Theorem \ref{extendedThm}.

In more detail, $(X,Z',10)$ is unexpected since $\dim [I(Z')]_{10}=\binom{10+3}{3}-c_{3,1,1,10}$ by Lemma \ref{adv} %Proposition \ref{BezoutTypeProp}
and $\binom{13}{3}-4c_{3,1,3,10}-5c_{3,1,1,10}=-1$, so we get
$$\dim [I(X)\cap I(Z')]_{10}=\dim [I(X(3^4,1^5))]_{10}=1 >0=\max\left(0,\binom{13}{3}-4c_{3,1,3,10}-5c_{3,1,1,10}\right)$$
$$=\max\left(0,\left(\binom{13}{3}-c_{3,1,1,10}\right)-4c_{3,1,3,10}-4c_{3,1,1,10}\right)$$
$$=\max\left(0,\dim [I(Z')]_{10}-4c_{3,1,3,10}-4c_{3,1,1,10}\right).$$ 
\end{example}

\begin{example}
Consider general lines $L_{1,1},\ldots,L_{1,4},L_{2,1},\ldots,L_{2,5}$.
Let $X=3L_{1,1}+3L_{1,2}+3L_{1,3}+L_{2,1}+\cdots+L_{2,5}$, $Z'=L_{1,4}$ and
let $Z$ be any 9 or more points of $L_{1,4}$ of multiplicity at most 3, where at least 9 of the points have multiplicity exactly 3.
Then as before, by Theorem \ref{ex:Singular list}(A) and Theorem \ref{extendedThm},
$(X,Z,10)$ and $(X,Z',10)$ are unexpected.
% Then $(X,Z,10)$ is unexpected, since $\dim [I(Z)]_{10}=\binom{10+3}{3}-c_{3,1,3,10}$ by Proposition \ref{BezoutTypeProp}, so
% $$\dim [I(X)\cap I(Z)]_{10}=\dim I[(X(3^4,1^5))]_{10}=1 > 0=\max\left(0,\binom{13}{3}-4c_{3,1,3,10}-5c_{3,1,1,10}\right)$$
% $$=\max\left(0,\left(\binom{13}{3}-c_{3,1,3,10}\right)-3c_{3,1,3,10}-5c_{3,1,1,10}\right)$$
% $$=\max(0,\dim [I(Z)]_{10}-3c_{3,1,3,10}-5c_{3,1,1,10}).$$ 
\end{example}

The preceding two examples work by replacing one of a set of lines coming from the schemes listed in Theorem \ref{ex:Singular list} by fat points.
Using the following proposition we can extend the construction given in the examples above by replacing more of the lines with points in a suitable way.

\begin{proposition}\label{extendedProp}
Let $L_1,\ldots,L_r,L_1',\ldots,L_s'$ be general lines in $\P^3$.
Let $(X+Z',\varnothing,t)$ correspond to one of the systems (A), (B), (C) or (D)  listed in Theorem \ref{ex:Singular list},
where $X=m_1L_1+\cdots+m_rL_r$ and $Z'=m_1'L_1'+\cdots+m_s'L_s'$
(so $X+Z'$ is a partition of one of the fat line schemes enumerated in the theorem). 
Then $u(X+Z',\varnothing,t)>u(Z',\varnothing,t)=0$ (and hence $(X,Z',t)$ and
$(X,Z,t)$ as given in Theorem \ref{extendedThm} are unexpected),
unless $X=\varnothing$ or, in case (A), $Z'$ includes all four lines of multiplicity 3, or,
in case (D), $Z'$ includes four or more lines of multiplicity 6.
\end{proposition}

\begin{proof}
It is enough in each case to show that $u(Z',\varnothing,t)=0$.
This is equivalent to showing $\adim (Z',\varnothing,t)=\edim (Z',\varnothing,t)$,
for which it is sufficient by semi-continuity to replace the general lines $L_i'$
by any choice of lines such that we get $\adim (Z',\varnothing,t)=\edim (Z',\varnothing,t)$.
So using Macaulay2 \cite{M2} we checked $\adim (Z',\varnothing,t)=\edim (Z',\varnothing,t)$ for each $Z'$
using random lines.
\end{proof}

%\noindent{\color{MyDarkGreen} !!!!!!!!!!!!!!! We excluded (D) from Proposition \ref{extendedProp} 
%since we have not yet gotten our calculations for (D) to run to completion.}

\begin{example}
Consider general lines $L_1,\ldots,L_6$.
Let $X=4L_1$, $Z'=3L_2+\cdots+3L_6$ and
let $Z$ be any 55 or more points of $L_2\cup\cdots\cup L_6$, each of multiplicity at most 3, with at least 11 points
of multiplicity 3 on each line $L_2,\ldots,L_6$.
By Proposition \ref{extendedProp}(B), we have $u(X+Z',\varnothing,12)>u(Z',\varnothing,12)$ 
so $(X,Z',12)$ and $(X,Z,12)$ are unexpected. 
(In fact, in this case, as noted in the proof of
Proposition \ref{extendedProp} for the case  (B), $u(Z',\varnothing,12)=0$, but 
$u(X+Z',\varnothing,12)=6$ by Theorem \ref{ex:Singular list}(B).)
\end{example}

\begin{example}\label{cubocubicEx}
We now present an example where 
$u(X+Z',\varnothing,t)>u(Z',\varnothing,t)=0$
but $X+Z'$ is not one of the cases listed in Theorem \ref{ex:Singular list}.
Let $X=X(3^{\times3})=3L_1+3L_2+3L_3$ and $Z'=3L_4$.
Since $Z'$ is a single fat line, we have $u(Z',\varnothing,10)=0$.
We claim $u(X+Z',\varnothing,10)=2$; given this, we have 
$u(X+Z',\varnothing,10)>u(Z',\varnothing,10)$ so by Theorem \ref{extendedThm}
$(X,Z',10)$ and $(X,Z,10)$ are unexpected.

To justify $u(X,\varnothing,10)=2$, note that by 
\cite{HarHir81} general lines of multiplicity 1 impose independent conditions on forms of degree $d$
when the virtual dimension is nonnegative for that degree, so we get $\adim(X(1^{\times4}),\varnothing,6)=56$ and 
$u(X_1(1^{\times4}),\varnothing,6))=0$. 

We now apply a certain cubo-cubic Cremona transformation (where the terminology cubo means that
that the transformation is defined by degree 3 forms, and the terminology cubic means the inverse 
transformation is also defined by degree 3 forms),
given by the linear system of cubics containing 4 general lines in $\P^3$. (This map is a special 
case of a Veneroni Cremona transformation, described below.)
It converts $(X_1(1^{\times4}),\varnothing,6)$ into $(X_1(3^{\times4}),\varnothing,10)$ (see \cite{DHRST18} or Example \ref{ExP3}),
so $\adim(X_1(3^{\times4}),\varnothing,10)=\adim(X(1^{\times4}),\varnothing,6)=56$ but now $\vdim(X_1(3^{\times4}),\varnothing,10)=54$
so $u(X+Z',\varnothing,10)=u(X_1(3^{\times4}),\varnothing,10)=2$.
\end{example}

\begin{example}\label{NiceLineExample}
	This example uses a linear system of surfaces of degree $20$,
	with four lines of multiplicity $6$. The unexpectedness in our example satisfies
$$u(X+Z',\varnothing,t)>u(Z',\varnothing,t)>0.$$

We have the following facts (justified below):
\begin{enumerate}
\item[(a)] $\adim(X_1(6^{\times5},1),\varnothing,20))=1$, $\vdim(X_1(6^{\times5},1),\varnothing,20))=-105$;
\item[(b)] $6\leq \adim(X_1(6^{\times5}),\varnothing,20))\leq16$, $\vdim(X_1(6^{\times5}),\varnothing,20))=-84$; and
\item[(c)] $\adim(X_1(6^{\times4}),\varnothing,20))=307$, $\vdim(X_1(6^{\times4}),\varnothing,20))=287$.
\end{enumerate}

Thus we have 
$$106=u(X_1(1)+X_1(6^{\times5}),\varnothing,20)>100\geq u(X_1(6^{\times5}),\varnothing,20)\geq 90,$$
so by Theorem \ref{extendedThm} we see that $(X_1(1),X_1(6^{\times5}),20)$ is unexpected.
In a similar but less surprising way, we see that 
$$100\geq u(X_1(6)+X_1(6^{\times4}),\varnothing,20)\geq90 >u(X_1(6^{\times4}),\varnothing,20)=20,$$
so by Theorem \ref{extendedThm} we see that 
$(X_1(6),X_1(6^{\times4}),20)$ is unexpected.

We now briefly justify the actual and virtual dimensions given above.

For (a), the triple 
$$(X_1(1^{\times5},6),\varnothing,8)$$ 
is expected since $\vdim((X_1(1^{\times5},6),\varnothing,8))=1$ and
we checked by computer that 
$$\adim(X_1(1^{\times5},6),\varnothing,8)=1.$$
As explained in \cite{DHRST18}, a Cremona transformation
due to J.~Todd (given by surfaces of degree $19$, passing through $5$ general lines with 
multiplicity $5$) converts $(X_1(1^{\times5},6),\varnothing,8)$ into $(X_1(6^{\times5},1),\varnothing,20)$.
Thus $\adim(X_1(6^{\times5},1),\varnothing,20)=1$ but now $\vdim(X_1(6^{\times5},1),\varnothing,20)=-105$.

For (b) we have $(X_1(6^{\times5}),\varnothing,20))$ which by applying 
the Todd transformation tranforms into
$(X_1(4^{\times5},10),\varnothing,20))$.
Note that $(X_1(4^{\times5},10),\varnothing,20))$ is obtained from $(X_1(2^{\times5},5),\varnothing,10))$ by 
doubling both the degree and 
the multiplicities. We get $\edim(X_1(2^{\times5},5),\varnothing,10)=6$ whereas by computer
(using random lines) we get $\adim(X_1(2^{\times5},5),\varnothing,10)=6$,
and hence $\adim(X_1(2^{\times5},5),\varnothing,10)=6$ holds for general lines.
Thus $\adim(X_1(6^{\times5}),\varnothing,20))\geq6$. Likewise, by computer 
(using random lines) we get $\adim(X_1(4^{\times5},10),\varnothing,20))=16$
and thus $\adim(X_1(4^{\times5},10),\varnothing,20))\leq16$ holds for general lines.

Finally, for (c) we have $\adim(X_1(2^{\times4}),\varnothing,12)=307$
(by computer, using random lines, semicontinuity and the fact that $\vdim(X_1(2^{\times4}),\varnothing,12)=307$).
A cubo-cubic Cremona transformation 
converts $(X_1(2^{\times4}),\varnothing,12)$ into $(X_1(6^{\times4}),\varnothing,20)$.
Thus $\adim(X_1(6^{\times4}),\varnothing,20)=307$ but $\vdim(X_1(6^{\times4}),\varnothing,20)=287$.
\end{example}

\section{Veneroni Cremona transformations}\label{Veneroni}

Up to now, we have considered unexpectedness only with respect to vanishing on 
linear spaces that are disjoint. To relax this condition, we look to
Veneroni maps. These are Cremona transformations of $\P^n$ studied by Veneroni \cite{veneroni}
(see \cite{Todd1930} for $n=4$, and also \cite{Blanch}, \cite{Snyder2} and \cite{DFHST} for additional exposition). 
Here we use them to give some examples of unexpectedness
with respect to vanishing on linear spaces that are not disjoint, thereby showing that the notion of
unexpectedness does not only make sense when the linear spaces are disjoint.

We recall how Veneroni maps are constructed.
The linear system of all degree $n$ forms vanishing on $n+1$ general linear subspaces 
$\Pi_1,\ldots  ,\Pi_{n+1}$ of codimension $2$ in $\P^n$ has dimension $n+1$ \cite{veneroni,DFHST}. Thus this linear system defines a
rational map $v_n:\P^n\dashrightarrow\P^n$, called the Veneroni map. 
In fact, $v_n$ is birational and its inverse is also a Veneroni map, defined by 
the degree $n$ forms vanishing on $n+1$ codimension $2$ linear subspaces $\Pi'_1,\ldots  ,\Pi'_{n+1}$ \cite{veneroni,DFHST}.
The base locus of $v_n$ consists of all the $\Pi_j$ together with the variety $R_n$ (also of codimension $2$) of 
all common transversal lines (i.e., all lines which intersect each $\Pi_j$) \cite{Snyder2,DFHST}.

For example, the quadratic Cremona transformation is obviously a Veneroni transformation $v_2$ on $\P^2$.
Recall that a Cremona transformation given by forms of degree $d$ whose inverse is given by forms of degree
$r$ is called a $d$-o--$r$-ic Cremona transformation. Thus
$v_3$ is a cubo-cubic Cremona transformation on $\P^3$ 
(which we applied above in Examples \ref{cubocubicEx} and \ref{NiceLineExample}) and 
$v_4$ is a quarto-quartic Cremona transformation on $\P^4$ (one such was used in \cite{Todd1930}).
 
To describe the map $v_n$ and its inverse, we will use the expression $dH-m_1\Pi_2-\cdots -m_{n+1}\Pi_{n+1}$
to denote the linear system of all forms of degree $d$ that vanish to order at least $m_i$ on $\Pi_i$ for each $i$.
So consider $v_n:\P^n_1\dashrightarrow\P^n_2$ and $v_n^{-1}:\P^n_2\dashrightarrow\P^n_1$, 
where we use subscripts to distinguish the source and target of $v_n$.
Let $H$ denote the linear system of all hyperplanes (i.e., linear forms) on $\P^n_1$
and let $H'$ denote the linear system of all hyperplanes on $\P^n_2$. 
Then $H'$ pulls back under $v_n$ to $nH-\Pi_1-\cdots -\Pi_{n+1}$. We indicate this by writing
$$H'=nH-\Pi_1-\cdots  -\Pi_{n+1}.$$
Moreover, the linear systems $(n-1)H-\Pi_1-\cdots  -\Pi_{n+1}+\Pi_j$ each have a unique
member, and this hypersurface is the inverse image under $v_n$ of $\Pi'_j$. We indicate this by writing
$$\Pi'_1=(n-1)H-\Pi_2-\cdots  -\Pi_{n+1},$$
$$\Pi'_2=(n-1)H-\Pi_1-\Pi_3-\cdots  -\Pi_{n+1},$$
$$\ldots$$
$$\Pi'_1=(n-1)H-\Pi_1-\cdots  -\Pi_{n}.$$
The inverse image under $v_n$ of $R_n'$ is $R_n$.
Likewise, with respect to $v_n^{-1}$ we have 
$$H=nH'-\Pi'_1-\cdots  -\Pi'_{n+1},$$
$$\Pi_1=(n-1)H'-\Pi'_2-\cdots  -\Pi'_{n+1},$$
$$\Pi_2=(n-1)H'-\Pi'_1-\Pi'_3-\cdots  -\Pi'_{n+1},$$
$$\ldots$$
$$\Pi_{n+1}=(n-1)H'-\Pi'_1-\cdots  -\Pi'_{n}.$$

We saw in Example \ref{cubocubicEx} for $n=3$ that by applying a Veneroni transformation to $S=6H-\Pi_1-\cdots-\Pi_4$
on $\P^3$ that $S'=v_3^{-1}(S)=10H-3(\Pi_1+\cdots+\Pi_4)$
is unexpected, even though the linear system $S$ had the expected dimension. 
We saw similar behavior in Example \ref{NiceLineExample} with respect to a different kind of Cremona transformation.
We now show some additional examples using Veneroni transformations that suggest this behavior may be fairly widespread.

\begin{example}\label{ExP3}
Take the linear system of surfaces $S=7H-\Pi_1-\cdots-\Pi_4$ in $\P^3$ of degree 7 vanishing on 7 general lines $\Pi_j$. 
The expected number of conditions imposed on forms of degree 7 by vanishing on $\Pi_1+\cdots+\Pi_4$ is 
$4c_{3,1,1,7}=32$.
Thus we get $\vdim(S)=\binom{7+3}{3}-32=88$, which is equal to $\adim(S)$ by \cite{HarHir81}.
In particular, the linear system $S$ has the expected dimension, but we now use a Veneroni transformation
to obtain a linear system $S'$ giving an unexpected hypersurface.

Pulling $S$ back by $v_3^{-1}$ gives the linear system 
$$S'=7(3H-\Pi_1-\cdots  -\Pi_4)-(2H-\Pi_2-\Pi_3-\Pi_4)-\cdots  -(2H-\Pi_1-\Pi_2-\Pi_3)=$$
$$(21-8)H-(7-3)\Pi_1-\cdots-(7-3)\Pi_4=13H-4\Pi_1-\cdots  -4\Pi_4.$$

Each $4\Pi_j$ imposes $c_{3,1,4,13}=120$ conditions on forms of degree 13, for a total of $4c_{3,1,4,13}=480$. 
Thus $u(S',\varnothing,13)=\adim(S')-\vdim(S')=\adim(S)-\vdim(S')=88 - (\binom{13+3}{4}-480)=88-(560-480)=8$, so $S'$ is unexpected.
\end{example}

We now want to consider similar cases for $n=4$ and $n=5$, but here the linear spaces $\Pi_j$ are not disjoint
so computing Hilbert polynomials is more involved. Suppose we have two schemes, $A$ and $B$, in $\P^n$, defined by homogeneous ideals
$I$ and $J$ in the homogeneous coordinate ring $R$ of $\P^n$. Let $C$ be the scheme defined by $I+J$
(this ideal need not be saturated even if $I$ and $J$ are, but this doesn't affect computing Hilbert polynomials).
Then we have the Mayer-Vietoris exact sequence
$$0\to \frac{R}{I\cap J}\to \frac{R}{I}\oplus \frac{R}{J}\to \frac{R}{I+J}\to 0$$
so we have $H_{A\cup B}=H_A+H_B-H_C$.
In the examples that follow, we want to compute the Hilbert polynomial for
$m_1\Pi_1+\dots+m_r\Pi_r$ in cases where $\Pi_i\cap\Pi_j$ is nonempty for 
$i\neq j$ but all triple intersections $\Pi_i\cap\Pi_j\cap\Pi_k$
are empty for $i\neq k\neq j$. We can do this recursively by taking 
$A=m_1\Pi_1$ and $B=m_2\Pi_2+\dots+m_r\Pi_r$, 
noting that triple intersections being empty implies
$C=(m_1\Pi_1\cap m_2\Pi_2)+\dots+(m_1\Pi_1\cap m_r\Pi_r)$.
Thus for $X=m_1\Pi_1+\dots+m_r\Pi_r$ we get recursively that
$$H_X=\sum_iH_{m_i\Pi_i}-\sum_{i\neq j}H_{m_i\Pi_i\cap m_j\Pi_j}.$$
 
\begin{example}\label{ExP4}
Take the hypersurface $S=7H-\Pi_1-\cdots-\Pi_5$ in $\P^4$. Because each pair of spaces $\Pi_i$ and $\Pi_j$
intersect in a point and there are $\binom{5}{2}$ such points, 
the expected number of conditions imposed on forms of degree 7 by vanishing on $\Pi_1+\cdots+\Pi_5$ is not 
$5c_{4,2,1,7}$, but rather $5c_{4,2,1,7}-\binom{5}{2}$.
Thus we get $\vdim(S)=\binom{7+4}{4}-5\binom{7+2}{2}+10=160$, which is equal to $\adim(S)$ (checked by Singular and Macaulay2),
hence as before $S$ has the expected dimension.

Pulling $S$ back by $v_4^{-1}$ gives the linear system 
$$S'=7(4H-\Pi_1-\cdots  -\Pi_5)-(3H-\Pi_2-\Pi_3-\Pi_4-\Pi_5)-\cdots  -(3H-\Pi_1-\Pi_2-\Pi_3-\Pi_4)=$$
$$(28-15)H-(7-4)\Pi_1-\cdots  -(7-4)\Pi_5=13H-3\Pi_1-\cdots  -3\Pi_5.$$

Each $3\Pi_j$ imposes $c_{4,2,3,13}=521$ conditions on forms of degree 13, for a total of $5c_{4,2,3,13}=2605$. But the 
nonempty intersection of each $3\Pi_j$ and $3\Pi_i$ reduces this by 36,
this being the (computer calculated) value of the Hilbert polynomial of 
$I_{3\Pi_j} + I_{3\Pi_i}$ in degree 13, giving $2605-\binom{5}{2}36=2245$ conditions.
(This also agrees with the value of the Hilbert polynomial of $I_{3\Pi_1+\cdots+3\Pi_5}$ in degree 13,
computed by Singular and Macaulay2 in a large positive characteristic.)
Thus $u(S',\varnothing,13)=\adim(S')-\vdim(S')=\adim(S)-\vdim(S')=160 - (2380-2245)=25$, so $S'$ is unexpected.
\end{example}

\begin{example}\label{ExP4b}
Now take the hypersurface $S=8H-\Pi_1-\cdots-\Pi_5$ in $\P^4$. 
The expected number of conditions imposed on forms of degree 8 by vanishing on $\Pi_1+\cdots+\Pi_5$ is 
$5c_{4,2,1,8}-\binom{5}{2}=5\binom{8+2}{2}-10=215$.
Thus we get $\vdim(S)=\binom{8+4}{4}-5\binom{8+2}{2}+10=280$, which is equal to $\adim(S)$ (checked by Singular and Macaulay2).

Pulling $S$ back by $v_4^{-1}$ gives the linear system 
$$S'=8(4H-\Pi_1-\cdots  -\Pi_5)-(3H-\Pi_2-\Pi_3-\Pi_4-\Pi_5)-\cdots  -(3H-\Pi_1-\Pi_2-\Pi_3-\Pi_4)=$$
$$17H-4\Pi_1-\cdots  -4\Pi_5.$$

Each $4\Pi_j$ imposes $c_{4,2,4,17}=1365$ conditions on forms of degree 17, for a total of $5c_{4,2,4,17}=6825$. But the 
nonempty intersection of each $4\Pi_j$ and $4\Pi_i$ reduces this by 100,
this being the (computer calculated) value of the Hilbert polynomial of 
$I_{4\Pi_j} + I_{4\Pi_i}$ in degree 17, giving $6825-\binom{5}{2}100=5825$ conditions.
(This also agrees with the value of the Hilbert polynomial of $I_{3\Pi_1+\cdots+3\Pi_5}$ in degree 17,
computed by Singular and Macaulay2 in a large positive characteristic.)
Thus $u(S',\varnothing,17)=\adim(S')-\vdim(S')=\adim(S)-\vdim(S')=280 - (\binom{17+4}{4}-5825)=280-(5985-5825)=280-160=120$, 
so $S'$ is unexpected.
\end{example}

\begin{example}\label{ExP5}
Take the hypersurface $S=8H-\Pi_1-\cdots-\Pi_6$ in $\P^5$. Again each pair of spaces $\Pi_i$ and $\Pi_j$
intersect, this time in a line, the expected number of conditions imposed on forms of degree 8 by vanishing on $\Pi_1+\cdots+\Pi_6$ is not 
$6c_{5,3,1,8}$, but rather $6c_{5,3,1,8}-\binom{6}{2}c_{5,1,1,8}=855$. This is confirmed using Singular \cite{Singular} 
to compute the Hilbert polynomial 
of the ideal of 6 random codimension 2 linear spaces, thus $\vdim(S)=\binom{5+8}{8}-855=432$. (Singular also gives
$\adim(S)=432$.)

Pulling $S$ back by $v_5^{-1}$ gives the linear system 
$$S'=8(5H-\Pi_1-\cdots  -\Pi_6)-(4H-\Pi_2-\Pi_3-\Pi_4-\Pi_6)-\cdots  -(4H-\Pi_1-\Pi_2-\Pi_3-\Pi_5)=$$
$$16H-3\Pi_1-\cdots  -3\Pi_6.$$

The value of the Hilbert polynomial of the ideal of $3\Pi_1+\cdots  +3\Pi_6$ in degree 16 is
$6c_{5,3,3,16}-\binom{6}{2}516=20106$, since 516 is the  (computer calculated) value of the Hilbert polynomial of 
$I_{3\Pi_j} + I_{3\Pi_i}$ in degree 16,
so we get $\vdim(S')=\binom{16+5}{5}-20106=243$.
But $\adim(S')=\adim(S)=432$, so we have $u(S',\varnothing,13)=\adim(S')-\vdim(S')=\adim(S)-\vdim(S')=432-243=189$, 
so $S'$ is unexpected.
\end{example}

\begin{example}\label{ExP5b}
Take the hypersurface $S=9H-\Pi_1-\cdots-\Pi_6$ in $\P^5$. Each pair of spaces $\Pi_i$ and $\Pi_j$
intersect in a line, so the expected number of conditions imposed on forms of degree 9 by vanishing on $\Pi_1+\cdots+\Pi_6$ is 
$6c_{5,3,1,9}-\binom{6}{2}c_{5,1,1,9}=1170$. This is confirmed using Macaulay2.
Thus $\vdim(S)=\binom{5+9}{9}-1170=832$. (Macaulay2 also gives $\adim(S)=832$.)

Pulling $S$ back by $v_5^{-1}$ gives the linear system 
$$S'=9(5H-\Pi_1-\cdots  -\Pi_6)-(4H-\Pi_2-\Pi_3-\Pi_4-\Pi_6)-\cdots  -(4H-\Pi_1-\Pi_2-\Pi_3-\Pi_5)=$$
$$21H-4\Pi_1-\cdots  -4\Pi_6.$$

The value of the Hilbert polynomial of the ideal of $4\Pi_1+\cdots  +4\Pi_6$ in degree 21 is
$6c_{5,3,4,21}-\binom{6}{2}1800=66036$, since 1800 is the  (computer calculated) value of the Hilbert polynomial of 
$I_{4\Pi_j} + I_{4\Pi_i}$ in degree 21. Hence $\vdim(S')=\binom{21+5}{5}-66,036=-256$. But $\adim(S')=\adim(S)=832$, so we have $u(S',\varnothing,21)=\adim(S')-\vdim(S')=\adim(S)-\vdim(S')=832+256=1088$,  so $S'$ is unexpected.
\end{example}

\begin{remark}
For convenience we used computer calculations in the examples above.
Here we demonstrate how to avoid using the computer by computing the Hilbert polynomial of the ideal of $4\Pi_1+\cdots  +4\Pi_6$
from Example \ref{ExP5b}. 
Let $I_1,\dots,I_6$ be the corresponding ideals, so each is generated by two general linear forms. By Lemma~\ref{adv} and a 
calculation, we see that the Hilbert polynomial of $R/I_i^4$ is  $\frac{5}{3} t^3 + \frac{10}{3} t + 1$ for each $1 \leq i \leq 6$. In 
particular the scheme $4 \Pi$ has degree 10. 

Notice that the intersection of any two $\Pi_i$ is a line, and the intersection of any three is empty. By \cite{MNP} Lemma~4.2, this means that the schemes meet ``very properly" (in the language of that paper), and so by Corollary 4.6 of that paper, the degree of the scheme-theoretic intersection of any two of the schemes $4 \Pi_i$ has degree $10^2 = 100$ and is supported on a line, and its saturated ideal is $I_{\Pi_1}^4 + I_{\Pi_2}^4$. 

We now want to find the Hilbert polynomial of this scheme, say $4 \Pi_1 \cap 4 \Pi_2$. Without loss of generality say $I_{\Pi_1} = (x_0,x_1)$ and $I_{\Pi_2} = (x_2,x_3)$. The artinian reduction of $R/I_{\Pi_1}^4$ can be taken to be $K[x_0,x_1]/(x_0,x_1)^4$ and analogously for $\Pi_2$, and the $h$-vectors are $(1,2,3,4)$. The tensor product of these two artinian rings is $K[x_0,x_1,x_2,x_3]/((x_0,x_1)^4 + (x_2,x_3)^4)$, which is the artinian reduction of $R/(I_{\Pi_1}^4 + I_{\Pi_2}^4)$. One can compute the $h$-vector to be $(1, 4, 10, 20, 25, 24, 16)$. Then ``integrating" this twice we obtain the Hilbert function of $4\Pi_1 \cap 4 \Pi_2$ to be the sequence
\[
1, 6, 21, 56, 116, 200, 300, 400, ... , 100t - 300
\]
so the Hilbert polynomial of $4\Pi_1 \cap 4 \Pi_2$ is $100t - 300$ (confirming the degree to be 100). 

From the Mayer-Vietoris sequence
\[
0 \rightarrow R/(I \cap J) \rightarrow R/I \oplus R/J \rightarrow R/(I+J) \rightarrow 0
\]
we see that the Hilbert polynomial of $4 \Pi_1 + 4 \Pi_2$ is 
\[
2 \left [ \frac{5}{3} t^3 + \frac{10}{3} t + 1 \right ] - [100 t - 300] = \frac{10}{3} t^3 - \frac{280}{3} t + 302.
\]
Putting together all six components, and recalling that no two of the intersections meet, we 
obtain the Hilbert polynomial of $4 \Pi_1 + \dots + 4 \Pi_6$ to be
\[
6 \left [ \frac{5}{3} t^3 + \frac{10}{3} t + 1 \right ] - 15 [100 t - 300] = 10t^3 - 1480t + 4506.
\]
In particular, consider $t = 21$. We obtain the value of this Hilbert polynomial to be 
\[
10(21)^3 - 1480(21) + 4506 = 66036.
\]
\end{remark}

\begin{remark}
We have seen by applying a Veneroni transformation to $S=(n+3)H-\Pi_1-\cdots-\Pi_{n+1}$
on $\P^n$ that $S'=v_n^{-1}(S)=(3n+1)H-3(\Pi_1+\cdots+\Pi_{n+1})$
is unexpected for 
$n=3$ (Example \ref{cubocubicEx}),
$n=4$ (Example \ref{ExP4}) and
$n=5$ (Example \ref{ExP5}), and for
$S=(n+4)H-\Pi_1-\cdots-\Pi_{n+1}$ that
$S'=v_n^{-1}(S)=(3n+2)H-4(\Pi_1+\cdots+\Pi_{n+1})$ is unexpected for
$n=3$ (Example \ref{ExP3}),
$n=4$ (Example \ref{ExP4b}) and
$n=5$ (Example \ref{ExP5b}).
It would be interesting to understand
more generally for which values of $t$ and $n$ the system $v_n^{-1}(tH-\Pi_1-\cdots-\Pi_{n+1})$ is unexpected.
\end{remark}

\end{document}